\documentclass[psamsfonts]{amsart}
\usepackage{amssymb}

\usepackage{amsmath,amstext,amssymb,amsopn,amsthm}
\usepackage{amsmath,amssymb,amsthm}
\usepackage[mathscr]{eucal}
\usepackage{amssymb}
\usepackage{amsmath}
\usepackage{color}
\usepackage{tikz}
\usetikzlibrary{arrows}
\usepackage[unicode,bookmarks,colorlinks]{hyperref}
\hypersetup{
    linkcolor=brickred,
}

\definecolor{mahogany}{cmyk}{0, 0.77, 0.87, 0}
\definecolor{salmon}{cmyk}{0, 0.53, 0.38, 0}
\definecolor{melon}{cmyk}{0, 0.46, 0.50, 0}
\definecolor{yellowgreen}{cmyk}{0.44, 0, 0.74, 0}
\definecolor{brickred}{cmyk}{0, 0.89, 0.94, 0.28}
\definecolor{OliveGreen}{cmyk}{0.64, 0, 0.95, 0.40}
\definecolor{RawSienna}{cmyk}{0, 0.72, 1.0, 0.45}
\definecolor{ZurichRed}{rgb}{1, 0, 0} 

\numberwithin{equation}{section}

\newtheorem{thm}{Theorem}[section]        
\newtheorem{cor}{Corollary}[section]
\newtheorem{lem}{Lemma}[section]
\newtheorem{prop}{Proposition}[section]
\newtheorem{rmk}{Remark}[section]

\newtheorem{question}{Question}[section]

\newcommand{\R}{\mathbb{R}}                  
\newcommand{\Rd}{\R^d}               
\newcommand{\ioRd}{\int_{\Rd}}              
\newcommand{\iogRd}[1]{\int_{\R^{#1d}} }    
\newcommand{\set}[1]{ \left\{#1\right\} }
\newcommand{\mysum}[3]{\sum\limits_{#1=#2}^{#3}}          
\newcommand{\myprod}[3]{\prod\limits_{#1=#2}^{#3}}

\newcommand{\wh}[1]{\widehat{#1}}              

\newcommand{\abs}[1]{\left|#1\right|}
\newcommand{\mymax}[2]{\max\limits_{#1=1,..,#2}}
\newcommand{\tgo}{t\downarrow 0}
\newcommand{\F} {(-\Delta)^{ \frac{\alpha}{2} } }
\newcommand{\Prob} {\mathbb{P} }
\newcommand{\palp}{p_{t}^{(\alpha)}}
\newcommand{\pH}{ p_{t}^{H_{V}}} 

\begin{document}

\title[Heat content of Schr\"odingier operators on $\Rd$]{Heat content and small time asymptotics for Schr\"odinger operators on $\Rd$}
\author{Luis Acu\~na Valverde}\thanks{Both authors supported in part by NSF Grant
\#0603701-DMS, PI. R. Ba\~nuelos}
\address{Department of Mathematics, Purdue University, West Lafayette, IN 47907, USA}
\email{lacunava@math.purdue.edu}

\author{Rodrigo Ba\~nuelos}
\address{Department of Mathematics, Purdue University, West Lafayette, IN 47907, USA}
\email{banuelos@math.purdue.edu}

\maketitle


\begin{abstract} 
This paper studies  the {\it heat content} for  Schr\"odinger operators of the fractional Laplacian $\F$, $0<\alpha\leq 2$ in $\Rd$, $d\geq 1$.   Employing probabilistic and analytic techniques, a small time asymptotic expansion formula is given and the {\it ``heat content invariants"} are identified.    These results are new even in the case of the Laplacian, $\alpha=2$.  
\end{abstract}

\section{introduction}

Let $0<\alpha\leq 2$ and consider $X=\set{X_t}_{t\geq0}$ a rotationally  invariant  $\alpha$-stable process whose transition densities $p_t^{(\alpha)}(x,y)=p_t^{(\alpha)}(x-y)$ are uniquely determined by their Fourier transform (characteristic function) and which are given by 
\begin{equation}\label{CharF.Proc}
e^{-t\abs{\xi}^{\alpha}}=\mathbb{E}^{0}[e^{-\dot{\iota} \xi\cdot X_t }]=\ioRd e^{-\dot{\iota}y\cdot\xi}\palp(y)dy,
\end{equation}
for all $t>0$, $\xi\in\Rd$, $d\geq 1$. Henceforth, $\mathbb{E}^x$ will represent the expectation of the process  starting at $x$.

For the purposes of this paper, we need to take into consideration both the spectral and integral definition for the  infinitesimal generator associated to $X$, denoted here by $H_{\alpha}=\F$. In the spectral theoretic sense, $\F$ is a  positive and self--adjoint linear operator   with domain
$$\set{f \in L^2(\Rd ) : \abs{\xi}^{\alpha}\wh{f}(\xi)\in L^2(\Rd )}$$
satisfying
\begin{equation}\label{Fourierstable}
 \wh{\F}f(\xi) = |\xi|^{\alpha}\wh{f}(\xi),
\end{equation}
where $\wh{f}$ denotes the Fourier transform of $f$.  Moreover, for $f \in\mathcal{S}(\Rd)$, where 
$\mathcal{S}(\Rd)$ is the set of  rapidly decreasing smooth functions,
we have   
$$\F f(x)=\left.\frac{d}{dt} e^{-t\F }f(x)\right\vert_{t=0},$$ where 
$$e^{-t\F }f(x)=\mathbb{E}^x[f(X_t)]=\ioRd \palp(x,y)f(y)dy$$ is the heat semigroup generated by $X$. On the other hand,  $\F$ can also be 
expressed in the integral form 
\begin{equation*}\label{fractsmoothfunct}
\F f (x) = C_{d,\alpha}\ioRd \frac{f (x) -f(y)}{|x-y|^{d+\alpha}}dy,
\end{equation*}
where $C_{d,\alpha}>0$ is a normalizing constant and the integral is understood in  the principal value sense. The last expression  allows
us to rewrite  the Dirichlet form associated to $\F$(see \cite{Fuk} for further details) 
\begin{equation}\label{DirichletF}
\mathcal{E}_{\alpha}(f)=\langle \F f, f\rangle = \int_{\Rd}\F f(x)f(x)dx
\end{equation}
as
\begin{equation}\label{DirichletF2}
\mathcal{E}_{\alpha}(f)=\frac{C_{d,\alpha}}{2}\int_{\Rd}\int_{\Rd}
\frac{\abs {f (x) -f(y)}^2}{|x-y|^{d+\alpha}}dxdy.
\end{equation}
Notice  that  when $\alpha=2$, due to  integration by parts,
we have
\begin{equation}\label{Laplaceform}
\mathcal{E}_2(f)=\int_{\Rd}(-\Delta f)(x)f(x)dx=\int_{\Rd}\abs{\nabla f(x)}^2dx, 
\end{equation}
which is the classical Dirichlet form of the Laplacian.  

Let $V\in L^{\infty}(\Rd)\cap L^{1}(\Rd)$.  The linear operator
$H_V=\F +V$, known as the Fractional Schr\"{o}dinger operator, is self-adjoint and defined similarly as the infinitesimal generator of the heat semigroup,  
$$e^{-tH_V}f(x)=\mathbb{E}^x[e^{-\int_{0}^{t}V(X_s)ds}f(X_t)],$$
for $f \in \mathcal{S}(\Rd).$ The heat kernel of $e^{-tH_V}$ is given by the 
Feymann-Kac formula (see \cite{Hir}, \cite{Kal} and  \cite{Sim})
\begin{equation}\label{Fey.Kac.f}
\pH(x,y)=\palp(x,y)\mathbb{E}_{x,y}^t\left[e^{-\int_{0}^{t}
V(X_{s})ds}\right], 
\end{equation} 
where $\mathbb{E}_{x,y}^t$ denotes the expectation with respect to the stable process (bridge) starting at $x$ and conditioned to be at $y$ at time $t$. 

With $H_{\alpha}$, $H_V$ and their heat kernels properly introduced,
we now proceed to consider  the heat trace for Schr\"{o}dinger operators, which is defined by
\begin{equation*}
Tr\left(e^{-tH_{V}}-e^{-tH_{\alpha}}\right)=\int_{\Rd}\left(p_t^{H_V}(x,x)-p_t^{(\alpha)}(x,x)\right)dx.
\end{equation*}
We set
\begin{equation}\label{htRd}
\mathcal{T}_V^{
(\alpha)}(t)=\frac{Tr\left(e^{-tH_{V}}-e^{-tH_{\alpha}}\right)}{p_t^{(\alpha)}(0)}.
\end{equation}

The small time asymptotic expansion for the Schr\"odinger operator corresponding to the case $\alpha=2$  (that is, the behavior of the quantity $\mathcal{T}_V^{
(2)}(t)$ as $t\downarrow 0$) has been extensively studied in the literature for many years by many authors.  One reason for this interest is its  connections and applications to spectral and scattering theory.  It is well known that there is an asymptotic expansion in  powers of $t$  and that the coefficients, known as the {\it ``heat invariants"}, encode rich information on scattering poles and properties of the potential  $V$.  (See for example, van den Berg \cite{vanden1},  McKean and Moerbeke \cite{MckMoe} and Melrose \cite{Mel1, Mel2}.)  For this reason, there has been a great deal of interest in obtaining explicit expressions for the coefficients in the expansions under suitable (but general  enough) assumptions on the potentials.  Following these works, 
S\'a Barreto and the second author proved the existence of an asymptotic expansion as $\tgo$ and gave a formula for the {\it ``heat invariants"} in terms of quantities involving the Fourier transform of the potential $V \in\mathcal{S}(\Rd)$.  This expansion allows for the computation of several coefficients and this in turn gives information on scattering poles for the potential $V$; see \cite[pp. 2162--2163]{BanSab} for details.  Applications of the  techniques in \cite{BanSab} for Schr\"odinger operators over compact Riemannian manifolds are given in Donnelly \cite{Don}.  In \cite{BanYil}, Yildirim and the second author proved a second order expansion for \eqref{htRd} as $\tgo$ valid for  all $0<\alpha\leq 2$ by imposing  a H\"older continuous condition on the potential $V$ similar to that imposed in \cite{vanden1} for the case $\alpha=2$.  In \cite{Acu}, the first author combined the techniques in \cite{BanSab} with probabilistic techniques to derive a general small time asymptotic expansion for $\mathcal{T}_V^{(\alpha)}(t)$ valid  for all $0<\alpha\leq  2$.  While this expansion is similar to that in \cite{BanSab} for $\alpha=2$, the formula for the  {\it $\alpha$-heat invariants} for $0<\alpha<2$ contains some rather complicated probabilistic quantities that are quite difficult to compute. Nevertheless, the expansion in \cite{Acu} permits the computation of several {\it heat invariants}  for the general $\alpha's$  as in the case of \cite{BanSab} for $\alpha=2$.   


While by no means complete (many questions concerning scattering theory remain completely open) our current understanding of trace asymptotics for the Schr\"odinger operator for the fractional Laplacians has greatly improved in recent years.  This is in start contrast to trace asymptotic for the Dirichlet fractional Laplacian in smooth bounded domains of $\Rd$ where progress has been slow.  
If $D\subset \Rd$ is a domain of finite volume and $p_t^D(x, y)$ is the heat kernel for the Laplacian in $D$ with Dirichlet boundary condition, then the following quantity 
\begin{equation*}\label{tracedef}
Z_{D}(t)=\int_Dp_t^D(x,x)dx,
\end{equation*}
is known as the {\it heat trace} of the Dirichlet heat semigroup for the domain  $D$. As in the case of the Schr\"odinger semigroups on $\Rd$, this quantity has been extensively investigated in the literature.  In 1954, S. Minakshiusundaram \cite{Min} proved that if $D\subset\R^d$ is a bounded domain with  smooth boundary (his result is  for general manifolds with boundaries), then there are constants $c_j(D)$ such that for all $N\geq 3$,
\begin{equation}\label{min}
Z_D(t)=\frac{1}{(4\pi t)^{d/2}}\left\{|D| -\frac{\sqrt{\pi}}{2}|\partial D|\,t^{1/2}+\sum_{j=3}^N c{_j}(D) t^{j/2}+ \mathcal{O}(t^{{(N+1)}/2})\right\}, 
\end{equation} 
as $t\downarrow 0.$  The {\it heat invariants} $c{_j}(D)$ have also been the source of intense interest for many years, specially following the foundational work of M. Kac \cite{Kac} and McKean and Singer  \cite{MckSin}.  In \cite{BanKul} and \cite{BanKulSiu}, a second order expansion is computed for the Dirichlet fractional Laplacian valid for all $0<\alpha\leq 2$.  However, a general asymptotic expansion for the trace of fractional Laplacian similar to \eqref{min} remains an interesting open problem. 

There are other spectral functions whose asymptotic expansions similarly encode important geometric information for  the domain $D$. One of these is the  {\it heat content} defined by 
\begin{equation*}\label{heatcontent}
Q_D(t)=\int_{D}\int_{D}p_t^D(x,y)dydx.
\end{equation*}
It represents the total amount of heat in the domain $D$ by  time $t$.  Like the trace, it has an asymptotic expansion of the form 
\begin{equation}\label{heatcontentasymptotic}
Q_D(t)=\sum_{j=0}^N a_j(D)t^{j/2} + \mathcal{O}(t^{{N+1)}/2})
\end{equation}
and the first few coefficients ({called \it heat content invariants}) have been calculated.  (See van den Berg and Gilkey \cite{vandenGil} and van den Berg, Gilkey, Kirsten and Kozlov \cite{vandenGilKirKoz} for more on the expansion and the calculation of  coefficients.) 
The following result was proved by van den Berg and Le Gall in  \cite{van2} for smooth domains $D\subset \R^d$, $d\geq 2$. 


\begin{equation}\label{VandenberHeatContBM}
Q_{D}(t)=|D|-\frac{2}{\sqrt{\pi}}|\partial D|t^{1/2}+\left(2^{-1}(d-1)\int_{\partial D}H(s)ds\right) t +\mathcal{O}(t^{3/2}),  
\end{equation}
as $t\downarrow 0$.  Here, $H(s)$ denotes the mean curvature at the point $s\in \partial D$.  
For more on the heat content asymptotics and its connections to the eigenvalues (spectrum)  of the Laplacian in the domain $D$, we refer the reader to van den Berg, Dryden and Kappeler \cite{vandenDryKap} and the many references to the literature contained therein.

\begin{question}
Is there an expansion similar to \eqref{heatcontentasymptotic} for  the Dirichlet semigroup of stable processes for  smooth bounded domains $D$, and can one compute the first few coefficients? In particular, is there a version of \eqref{VandenberHeatContBM}  for stable processes?   
\end{question}
 As of now, these too remain challenging open questions.  Even obtaining a second order asymptotic seems to be very challenging.   We remark that the tools to show the existence of the asymptotic expansion  \eqref{VandenberHeatContBM} depend strongly  on the fact that Brownian motion $B=\set{B_t}_{t\geq0}$ on  $\R^d$ is obtained by taking 
$d$--independent copies of a 1-dimensional Brownian motion as its coordinates.  This facilitates  many calculations in the above expansions,  often reducing matters to one dimensional problems; see for example \cite{vanden2}.  Unfortunately, these type of arguments completely fail for stable processes.  To understand more what these difficulties entail, we refer the reader to \cite{BanKul} and \cite{BanKulSiu} where similar issues have to be confronted for trace asymptotics. 

The above mentioned results on the trace of Schr\"odinger operators on $\Rd$ and those for the Dirichlet semigroup motivate the study of what we will call  {\it ``the heat content for Schr\"odinger semigroups"} and which we define by 
\begin{align}\label{heatcontent}
Q_V^{(\alpha)}(t)&=\ioRd\ioRd \set{\pH(x,y)-\palp(x,y)}dxdy \\ \nonumber\\\nonumber 
     &=\ioRd \ioRd \palp(x,y)\mathbb{ E}^{t}_{x,y}\left[ e^{-\int_{0}^{t}V(X_s)ds }-1\right] dxdy.
\end{align}
Notice that the second equality comes from \eqref{Fey.Kac.f}.  To the best of our knowledge, this quantity has not been studied in the literature before even in the case of the Laplacian.  

Before stating our results, we elaborate further on the name {\it``heat content"}.   Recall that the heat kernels for the semigroups $e^{-tH_{\alpha}}$ and  $e^{-tH_V}$ of the operators $\F$ and $\F +V$, respectively, 
satisfy the heat equations
\begin{equation*}
\frac{d}{dt}\palp(x,y)=-(-\Delta)_{x}^{\frac{\alpha}{2}}\palp(x,y),
\,\, t>0,\,\, (x,y)\in\R^{2d},
\end{equation*}
with initial condition
$$
 p_{0}^{(\alpha)}(x,y)= \delta(x-y), 
 $$
and
\begin{equation*}
\frac{d}{dt}\pH(x,y)=-[(-\Delta)_{x}^{\frac{\alpha}{2}} + V(x)]\pH(x,y), \,\, t>0,\,\, (x,y)\in\R^{2d}
\end{equation*} 
with $$p_{0}^{H_{V}}(x,y)= \delta(x-y).$$
Consequently, the function $$u(t,x,y)=\pH(x,y)-\palp(x,y)$$ satisfies
\begin{equation*}
\frac{d}{dt}u(t,x,y)= -(-\Delta)_x^{\alpha/2}u(t,x,y)-V(x)\pH(x,y), \,\,\,  (t, x, y)\in (0,+\infty)\times\R^{2d}
\end{equation*} 
with initial condition 
$$u(0,x,y)=0, \,\,\, (x, y)\in  \R^{2d}.
$$

From $\eqref{Fey.Kac.f}$, we observe that when $V\leq 0$, we have 
$u(t,x,y)\geq 0$ so that we can interpret $u(t,x,y)$ as a temperature function that reflects the excess of heat generated by the potential $V$ at time $t>0$ at the point $(x,y)$. Similarly, when $V\geq 0$, $u(t,x,y)\leq 0$ , which can be interpreted as a loss of heat.
Likewise, $$U(t,x)=\int_{\Rd}u(t,x,y)dy$$ can be regarded as a temperature function defined on $\Rd$ which lets us interpret $Q_V^{(\alpha)}(t)$ as the amount of heat that the Euclidean space $\R^{d}$ has gained, or lost, by time $t>0$ with respect to the potential $V$. One of our goals is to compare the expansion as $\tgo$ for the heat content $Q_V^{(\alpha)}(t)$ to that of  $\mathcal{T}_V^{(\alpha)}(t)$ proved 
in \cite{BanSab} ($\alpha=2$ case), \cite{BanYil} and \cite{Acu} ($0<\alpha\leq 2$ case). 

 We proceed to state our main results. The first  two theorems correspond to the results for the trace proved in  \cite{vanden1} for $\alpha=2$ and in \cite{BanYil} for $0<\alpha<2$.   The first  theorem  provides  the first term whereas the second theorem yields a second  order expansions under the assumption of  a H\"older continuity on the potential $V$.  Both theorems provide uniform bounds for the remainder term  for all positive times.

\begin{thm}\label{1term}
\hspace*{20mm}
\begin{enumerate}
\item[(i)] Assume $V\in L^{\infty}(\Rd)\cap L^{1}(\Rd)$. Then if  $V:\Rd \rightarrow (-\infty,0]$, we have for all $t>0$ that
$$-t\ioRd V(x)dx\leq Q_V^{(\alpha)}(t) \leq -t\ioRd V(x)dx \left( 1+\frac{1}{2}t||V||_{\infty}e^{t||V||_{\infty}}\right).$$
The last inequality implies that
\begin{equation*}
Q_V^{(\alpha)}(t)=-t\ioRd V(x)dx + \mathcal{O}(t^2), 
\end{equation*}
as $\tgo$.
\bigskip

\item[(ii)]For $V\in L^{\infty}(\Rd)\cap L^{1}(\Rd)$, we obtain for all $t>0$
$$\abs{Q_V^{(\alpha)}(t)+t\ioRd V(x)dx}\leq t^2||V||_1||V||_{\infty}e^{t||V||_{\infty}}. $$
In particular, 
$$Q_V^{(\alpha)}(t)=-t\ioRd V(x)dx+ \mathcal{O}(t^2),$$ as $\tgo$.
\end{enumerate}
\end{thm}

\begin{thm}\label{Holdercont}
Suppose $V\in L^{\infty}(\Rd)\cap L^{1}(\Rd)$. Assume that $V$ is also uniformly H\"{o}lder continuous of order $\gamma$.  That is, there exists a positive constant $M$ such that $\abs{V(x)-V(y)}\leq M|x-y|^{\gamma}$, for all $x,y \in \Rd$, with 
$0<\gamma<\min\set{1,\alpha}$, $0<\alpha\leq 2$. Then, for all $t>0$
$$ \abs{ Q_V^{(\alpha)}(t)+t \ioRd V(x)dx-\frac{t^2}{2}\ioRd V^2(x) dx}\leq 
C(\gamma,\alpha)||V||_1 \left( ||V||_{\infty}^2 e^{ t||V||_{\infty}}t^{3}+t^{\frac{\gamma}{\alpha}+2}\right).$$
In particular,
\begin{equation*}
Q_V^{(\alpha)}(t)=-t \ioRd V(x)dx+\frac{t^2}{2}\ioRd V^2(x) dx +\mathcal{O}(t^{\frac{\gamma}{\alpha}+2}),
\end{equation*}
as $\tgo$.
\end{thm}

It is interesting to note here  that in \cite{BanYil}, it is shown that
\begin{equation*}
\mathcal{T}_V^{(\alpha)}(t)=-t\int_{\Rd}V(\theta)d\theta + \frac{t^2}{2}\int_{\Rd}V^2(\theta)d \theta + \mathcal{O}(t^{\frac{\gamma}{\alpha}+2}),
\end{equation*}
as $\tgo$ under the same conditions of Theorem \ref{Holdercont}. Thus, under the assumption of H\"older continuity we cannot distinguish between $Q_V^{(\alpha)}(t)$ and $\mathcal{T}_V^{(\alpha)}(t)$,  as $\tgo$ as the scone order asymptotic expansion.  In order to see the difference in these quantities for $t\downarrow 0$,  we need to assume extra regularity conditions on $V$ and go further in the expansion. 

Our third result in this paper is a general asymptotic expansion in powers of $t$ for potentials $V\in \mathcal{S}(\Rd)$ with an explicit form for the coefficients.  In order to avoid the introduction of more complicated notation at this point, we postpone the result to Theorem \ref{Generalexpansion} in \S \ref{sec:GeneralCoef}.  A special case of Theorem \ref{Generalexpansion} where we can compute quite explicitly all the coefficients is the following theorem.
 
\begin{thm}\label{5termexp} Let $V\in \mathcal{S}(\Rd)$ and  $0<\alpha\leq 2$.
Then 
\begin{align*}
Q_V^{(\alpha)}(t)=&-t\ioRd V(\theta)d\theta +\frac{t^2}{2!}\ioRd V^2(\theta)d\theta
-\frac{t^3}{3!}\left( \ioRd V^3(\theta)d\theta +\mathcal{E}_{\alpha}(V)\right)\\
&+\frac{t^4}{4!}\left(\ioRd V^4(\theta)d\theta+2\ioRd V^2(\theta)\F V(\theta)d\theta +\int_{\Rd}\abs{\F V(\theta)}^2d\theta\right)\\
&-\frac{t^5}{5!}\Biggl(\int_{\Rd}V^5(\theta)d\theta+
2\int_{\Rd}V^3(\theta)\F V(\theta)d\theta 
+2\int_{\Rd}V^2(\theta)\F_2 V(\theta)d\theta \\ 
&+\int_{\Rd}V(\theta)\abs{\F V(\theta)}^{2}d\theta + \mathcal{E}_{\alpha}\left(\F V\right)+\mathcal{E}_{\alpha}\left( V^2\right)\Biggr)+ \mathcal{O}(t^6),
\end{align*}
as $\tgo$. Here, $\mathcal{E}_{\alpha } $ is the Dirichlet form as defined  in \eqref{DirichletF2} and \eqref{Laplaceform} whereas $\F_2$ is defined to be
$\F \circ \F.$ 
\end{thm}


The expansion above  enables us to comment on the  similarities and differences 
between the heat trace and the heat content.  We start with the case $\alpha=2$. It is proved in \cite{BanSab} that
\begin{align*}    
\mathcal{T}_V^{(2)}(t)&=-t\ioRd V(\theta)d\theta + \frac{t^{2}}{2!}\ioRd V^2(\theta)d\theta -\frac{ t^{3}}{3!}\left(\ioRd V^3(\theta)d\theta + \frac{1}{2}\mathcal{E}_2(V)\right)\\ \nonumber
&+\frac{t^4}{4!}\left(\ioRd V^4(\theta)d\theta +2\ioRd  V(\theta)\abs{\nabla V(\theta)}^2d\theta+\frac{1}{5}\ioRd \abs{\left(-\Delta\right)V(\theta)}^2d\theta\right) \\ \nonumber
&-\frac{t^5}{5!}\int_{\Rd}\Biggr (V^5(\theta)d\theta
+\frac{3}{42} \abs{\nabla (-\Delta) V(\theta)}^2 +
5 V^2(\theta)\abs{\nabla V(\theta)}^2 \\ 
&+\frac{15}{27} V(\theta)\abs{(-\Delta) V(\theta)}^2 +
\frac{4}{9}V(\theta)\left(\sum\limits_{i,j=1}^{d}\partial_{x_i}
\partial_{x_j}V(\theta)\right)^2
\Biggl)d\theta +
\mathcal{O}(t^ 6), 
\end{align*}
as $\tgo$.

From 
$$
\ioRd V^2(\theta)
(-\Delta V)(\theta)d\theta=2\int_{\Rd}V(\theta)|\nabla V(\theta)|^2d\theta
$$ and 
$$
\int_{\Rd}\abs{\nabla (-\Delta) V(\theta)}^2d\theta=
\int_{\Rd}(-\Delta V)(\theta)(-\Delta)_2 V(\theta)d\theta=
\mathcal{E}_2((-\Delta)V)
$$
and other similar identities,  we note  that by  Theorem \ref{5termexp} the same integrands are involved in the expansion of both the heat trace and the heat content and both expansions  behave similarly  as $\tgo$.
On the other hand, in the case $0<\alpha<2$, the  heat trace and the heat content have completely different behavior for small $t$. In fact, the asymptotic expansion provided in \cite{Acu} for  $\mathcal{T}_V^{(\alpha)}(t)$ is dimensional dependent,  unlike the situation of  $\alpha=2$. That is, the powers of $t$ in the expansion depend on the location of $\alpha$ relative to the dimension $d$.  One of the  strongest result obtained in \cite{Acu} is that 
\begin{align}\label{SR}
\mathcal{T}_V^{(\alpha)}(t)=&-t\ioRd V(\theta)d\theta + \frac{t^{2}}{2!}\ioRd V^2(\theta)d\theta - \frac{t^3}{3!}\ioRd V^3(\theta)d\theta \\ \nonumber
&-{\mathcal L}_{d,\alpha}\,t^{2+\frac{2}{\alpha}}\mathcal{E}_{2}(V) 
 + {\mathcal O}(t^4),
\end{align} 
for all $d\geq 1$ and $\frac{3}{2}<\alpha<2$ as $\tgo$, where ${\mathcal L}_{d,\alpha}>0$ is a constant depending on the $\alpha/2$-subordinator related to $X$. There, it is also shown that the same integrals involved  in the expansion for $\mathcal{T}_V^{(2)}(t)$ appear in different positions in the corresponding expansion for $\mathcal{T}_V^{(\alpha)}(t)$ according to the given $\alpha$ under consideration. With these observations and Theorem \ref{Generalexpansion} below, we conclude that the expansion of $Q_{V}^{(\alpha)}(t)$  gives information on the action of the operator $\F$  on the potential $V$ and produces functions of $t$ of the form $c t^n$, $n \in\mathbb{N}$, with explicit real numbers $c$. On the other hand, according to the results
in \cite{Acu}, $\mathcal{T}_V^{(\alpha)}(t)$ gives information on the action of $\Delta$ on $V$ and produces powers of $t$ of the form $c_{\alpha,d}t^{n+\frac{j}{\alpha}}$, 
$n,j\in \mathbb{N}$, where $c_{d,\alpha}$, although explicitly given, are as of now quite difficult compute  for general $d$ and $\alpha$. For more on this, we refer the reader to \cite{Acu}.

 The paper is organized as follows. In \S \ref{sec:heatscho}, we show that $Q_V^{(\alpha)}(t)$ is a well defined function for every bounded and integral potential $V$ and prove  a series of lemmas needed for  the proof of Theorem \ref{1term} and Theorem \ref{Holdercont}. In \S \ref{sec:smoothpot}, we prove the existence of a general expansion for $Q_V^{(\alpha)}(t)$ (Theorem \ref{Generalexpansion}) under the assumption that the potential is a rapidly decreasing smooth function.  This is done using  Fourier Transform techniques. Lastly, in \S \ref{sec:coeff}, we compute the first five terms in the expansion obtained in \S \ref{sec:smoothpot} which proves Theorem \ref{5termexp}. 

\section{proof of theorems \ref{1term} and \ref{Holdercont}.}\label{sec:heatscho}
 We start this section by proving  that $Q_V^{(\alpha)}(t)$
 given by  \eqref{heatcontent}  is a well--defined function for all $t \geq 0$  as long as $V$ is  bounded and integrable.
We begin by observing that the elementary inequality $\abs{e^z-1}\leq \abs{z}e^{\abs{z}}$ gives 
\begin{equation*}
\abs{Q_V^{(\alpha)}(t)} \leq e^{t||V||_{\infty}}\iogRd{2}\palp(x,y)\mathbb{E}_{x,y}^t\left[\int_{0}^{t}|V(X_s)|ds\right]dxdy.
\end{equation*}
Next, by Fubini's theorem and the properties of stable bridge (see \cite{BanYil}, \cite{Ber} and \eqref{stablebridge} below) the integral term in the right hand side of the above inequality equals 
\begin{align}\label{basic.cal}
\iogRd{2}\palp(x,y)\left(\int_{0}^{t}\ioRd\abs{V(z)}\frac{p_{t-s}^{(\alpha)}(x,z)p_{s}^{(\alpha)}(z,y)}{\palp(x,y)}dzds\right)dxdy &=  \\ \nonumber
\int_{0}^{t}\ioRd\abs{V(z)}
\left(\ioRd p_{t-s}^{(\alpha)}(x,z)dx\ioRd p_{s}^{(\alpha)}(z,y)dy\right) dzds &=
t||V||_1,
\end{align}
where we have used the well known facts that for all $x, z\in\Rd$ and $t>0$,
$\palp(x,z)=\palp(z,x)$ and $\ioRd \palp(x,z)dx=1.$ Thus, we conclude that the heat 
content satisfies $$\abs{Q_V^{(\alpha)}(t)}\leq t||V||_1e^{t||V||_{\infty}}.$$ Therefore, $Q_V^{(\alpha)}(t)$ is well-defined for all $t\geq 0$ and bounded on any interval $(0,T]$, $T>0$, provided $V\in L^{\infty}(\Rd)\cap L^{1}(\Rd).$   It is also worth noting here that the  previous argument together with 
Taylor's expansion of the exponential function (see \eqref{Tay.exp} below) show that
\begin{equation}\label{asympofQ}
Q_V^{(\alpha)}(t)=\sum\limits_{k=0}^{\infty}\frac{(-1)^k}{k!}
\iogRd{2}\palp(x,y)\mathbb{E}^{t}_{x,y}\left[\left( \int_{0}^{t}V(X_s)ds\right)^k\right] dxdy,
\end{equation}
where the sum is absolutely convergent for all $t > 0$.

It is advantageous at this point to give a different expression for the  equation \eqref{asympofQ} in terms of the stable bridge in order to obtain further formulas for the coefficients and  estimates for the remainders in the forthcoming sections.
Before proceeding, we  introduce some notation to conveniently express our formulas below. For $k\in \mathbb{N}$, we set
\begin{align}\label{heavynotation}
I_k &=\set{\lambda^{(k)}=(\lambda_{k},\lambda_{k-1},...\lambda_1)\in [0,1]^{k}: 0<\lambda_{k}<\lambda_{k-1}< ... < \lambda_{1}<1},  \\ \nonumber \\ \nonumber
d\lambda^{(k)}&=d\lambda_k d\lambda_{k-1}...d\lambda_1, \,\,\,\,\, 
z^{(k)}=(z_1,...,z_k)\in \R^{kd},\,\,\,\, 
dz^{(k)} =dz_k...dz_1,\\ \nonumber \\ \nonumber
V_{k}\left(z^{(k)}\right)&=V_k(z_1,...,z_k)=\myprod{i}{1}{k}V(z_i), \,\,\,\,\,\,\, 
p\left(t,z^{(k)}\right)=\myprod{j}{1}{k-1}p_{t(\lambda_{j}-\lambda_{j+1})}^{(\alpha)}(z_{j},z_{j+1}).
\end{align}

It is well known \cite{Tib}  that
\begin{equation}\label{symmetric}
\left(\int_{0}^{1}\tilde{V}(s)ds\right)^k= k!\int_{I_k}\myprod{i}{1}{k}\tilde{V}(\lambda_i)d\lambda^{(k)},
\end{equation}
for any $\tilde{V}:[0,1]\rightarrow \R$ integrable.

\begin{lem}\label{Qbridge} 
For any $t>0$ and $J\geq 2$,
\begin{eqnarray*}
Q_V^{(\alpha)}(t)=-t\ioRd V(\theta)d\theta +
\mysum{k}{2}{J}(-t)^k
\int_{I_k}\iogRd{k}V_{k}\left(z^{(k)}\right)p\left(t,z^{(k)}
\right)dz^{(k)}d\lambda^{(k)}
+ R_{J+1}(t),
\end{eqnarray*}
where 
\begin{equation}\label{Rem2}
\abs{R_{J+1}(t)}\leq t^{J+1}||V||_1||V||_{\infty}^{J}e^{t||V||_{\infty}}.
\end{equation}
\end{lem} 

\begin{proof}
Set 
\begin{align*}
R_{J+1}(t)=\sum\limits_{k=J+1}^{\infty}\frac{(-1)^k}{k!}
\iogRd{2}\palp(x,y)\mathbb{E}^{t}_{x,y}\left[\left( \int_{0}^{t}V(X_s)ds\right)^k\right] dxdy.
\end{align*} 
It  is clear  by \eqref{basic.cal} that
\begin{align*}
\abs{ R_{J+1}(t)}&\leq\sum\limits_{k=J+1}^{\infty}\frac{(t||V||_{\infty})^{k-1}}{k!}
\iogRd{2}\palp(x,y)\mathbb{E}^{t}_{x,y}\left[ \int_{0}^{t}\abs{V(X_s)}ds\right] dxdy \\
&\leq t||V||_1 \sum\limits_{k=J+1}^{\infty}\frac{(t||V||_{\infty})^{k-1}}{k!}
\leq  t^{J+1}||V||_1||V||_{\infty}^{J}e^{t||V||_{\infty}}.
\end{align*}

On the other hand, by making a suitable change  of variables and appealing to \eqref{symmetric} we observe that
\begin{align*}
\mathbb{E}^{t}_{x,y}\left[\left( \int_{0}^{t}V(X_s)ds\right)^k\right]&=
t^{k}\mathbb{E}^{t}_{x,y}\left[\left( \int_{0}^{1}V(X_{ts})ds\right)^k\right] \\ \nonumber
&=k!t^{k}\mathbb{E}^{t}_{x,y}\left[\int_{I_{k}}V(X_{t\lambda_1})...V(X_{t\lambda_{k}})
d\lambda^{(k)}\right].
\end{align*}

We recall that the finite dimensional distributions of the stable bridge (see \cite{BanYil}, \cite{Ber} and references therein for details) are given by
\begin{align}\label{stablebridge}
&\Prob^t_{x,y}\left(X_{t\lambda_1} \in dz_1 , X_{t\lambda_2} \in dz_2 , . . . , X_{t\lambda_k} \in dz_k\right)\\\nonumber &=\frac{1}{\palp(x, y)}\myprod{j}{0}{k}p_{t(\lambda_{j}-\lambda_{j+1})}^{(\alpha)}(z_{j},z_{j+1})dz^{(k)},
\end{align}
where $z_0 =x$, $z_{k+1} = y$, $\lambda_0 = 1$, and $\lambda_{k+1} = 0$.
Hence, using the fact  that
$$\int_{\Rd}p_{t(1-\lambda_1)}^{(\alpha)}(x,z_1)dx=
\int_{\Rd}p_{t\lambda_k}^{(\alpha)}(z_k,y)dy=1,$$ the notation given in \eqref{heavynotation} and the finite 
distribution for the stable bridge given above, we conclude by Fubini's theorem  that
\begin{align}\label{GeneralTerm}
&\iogRd{2}\palp(x,y)\mathbb{E}^{t}_{x,y}\left[\left( \int_{0}^{t}V(X_s)ds\right)^k\right]dxdy\\\nonumber
 &=
k!t^{k}\int_{I_k}\iogRd{k}\myprod{i}{1}{k}V(z_i)\myprod{j}{1}{k-1}
p_{t(\lambda_{j}-\lambda_{j+1})}^{(\alpha)}(z_{j},z_{j+1})
dz^{(k)}d\lambda^{(k)} \\ \nonumber &=
k!t^{k}\int_{I_k}\iogRd{k}V_{k}\left(z^{(k)}\right)
p\left(t,z^{(k)}\right)dz^{(k)}d\lambda^{(k)}.
\end{align}
Therefore, the lemma follows from equation \eqref{asympofQ}.
\end{proof}

{\bf Proof of Theorem \ref{1term}}:
Setting $a=\int_{0}^{t} V(X_s)ds$ and $b=t||V||_{\infty}$, we observe that $-b\leq a<0$.
We use the elementary   inequality
$$ -a\leq e^{-a}-1\leq -a\left(1+\frac{1}{2}be^b\right),$$ to obtain
\begin{equation}\label{Vneg.asym} 
-\int_{0}^t V(X_s)ds \leq e^{-\int_{0}^{t}V(X_s)ds }-1 \leq -\int_{0}^t V(X_s)ds
\left(1+\frac{1}{2}t||V||_{\infty}e^{t||V||_{\infty}}\right).
\end{equation}
By taking expectations \,$\mathbb{E}^{t}_{x,y}$\, at both sides of \eqref{Vneg.asym}, multiplying through by $\palp(x,y)$, integrating on $\R^{2d}$ with respect to $x$ and $y$ and appealing to \eqref{basic.cal} where $|V|$ is replaced by  $-V\geq0$, we arrive at
$$-t\ioRd V(x)dx\leq Q_V^{(\alpha)}(t) \leq -t\ioRd V(x)dx \left( 1+\frac{1}{2}t||V||_{\infty}e^{t||V||_{\infty}}\right).$$ Thus, (i) follows. 

 By using \eqref{basic.cal} and \eqref{asympofQ}, we have 
\begin{align*}\label{asympQ&rem}
\abs{Q_V^{(\alpha)}(t)+t\ioRd V(x)dx}&\leq\mysum{k}{2}{\infty}\frac{1}{k!}
\iogRd{2}\palp(x,y)\mathbb{E}^{t}_{x,y}\left[\left( \int_{0}^{t}\abs{V(X_s)}ds\right)^k\right] dxdy \\ \nonumber
&\leq t||V||_1\mysum{k}{2}{\infty}\frac{1}{k!}t^{k-1}||V||^{k-1}_{\infty}
\leq t^2||V||_1||V||_{\infty}e^{t||V||_{\infty}},
\end{align*}
which in turn proves (ii). 
\bigskip

{\bf Proof of Theorem \ref{Holdercont}:}
We start by recalling two basic facts about the  $\alpha$-stable process $X$, $0<\alpha\leq 2$. First,
\begin{equation}\label{expofX1andscaling}
 \mathbb{E}^{0}\left[|X_1|^{\gamma}\right]<\infty,
\end{equation}
whenever $\gamma<\alpha<2$. As for $\alpha=2$, the 
above fact is also true, since in this case $\gamma=1$. Secondly, 
\begin{equation}\label{eqinlaw}
X_t=t^{1/\alpha}X_1,
\end{equation}
in law as we can see from the characteristic function \eqref{CharF.Proc}.

The H\"older continuity assumption  on $V$,  as we shall see in the next lemma,  enables  us to estimate  the second term in \eqref{asympofQ}.

\begin{lem} \label{2ndtermestimate}
Under the same assumptions on the potential V given in Theorem \ref{Holdercont}, we have
for all $t>0$ that
 $$\iogRd{2}\palp(x,y)\mathbb{E}^{t}_{x,y}\left[\left(\int_{0}^{t}V(X_s)ds\right)\right]dxdy= 2t^2\ioRd\abs{V(x)}^2dx + R(t),$$ where the remainder
 $R(t)$ satisfies 
 $$\abs{R(t)}\leq C_0(\alpha,\gamma)||V||_1t^{\frac{\gamma}{\alpha}+2}.$$
\end{lem}

\begin{proof}
We start by applying \eqref{GeneralTerm} with $k=2$, so that

\begin{align*}
&\iogRd{2}\palp(x,y)\mathbb{E}^{t}_{x,y}\left[\left(\int_{0}^{t}
V(X_s)ds\right)^2\right]dxdy=\\
&2t^2\int_{0}^{1}\int_{0}^{\lambda_1}\iogRd{2}V(z_1)V(z_2)
 p_{t(\lambda_1-\lambda_2)}^{(\alpha)}(z_2,z_1)dz_2dz_1d\lambda_2d\lambda_1.
\end{align*}
Now,
\begin{align}\label{bridge2}
&\iogRd{2}V(z_1)V(z_2)p_{t(\lambda_1-\lambda_2)}^{(\alpha)}(z_2,z_1)dz_1dz_2=
\\\nonumber &\iogRd{2}\left(V(z_1)-V(z_2)\right)V(z_2)p_{t(\lambda_1-\lambda_2)}^{(\alpha)}(z_2,z_1)dz_1dz_2\\\nonumber
&+
\iogRd{2}V^2(z_2) p_{t(\lambda_1-\lambda_2)}^{(\alpha)}(z_2,z_1)dz_1dz_2.
\end{align}

The second term on the right hand side of  equality  \eqref{bridge2} equals  $\ioRd V^2(x)dx$, whereas for the first term, by using  \eqref{expofX1andscaling},
\eqref{eqinlaw} and the H\"{o}lder continuity assumption on $V$,  we have 
\begin{align}\label{1remestimate}
&\abs{\iogRd{2}\left(V(z_1)-V(z_2)\right)V(z_2)p_{t(\lambda_1-\lambda_2)}^{(\alpha)}(z_2,z_1)dz_1dz_2}\\ \nonumber &\leq 
M\iogRd{2}|z_1-z_2|^{\gamma}|V(z_2)|p_{t(\lambda_1-\lambda_2)}^{(\alpha)}(z_2,z_1)dz_1dz_2\\ \nonumber 
&=M \ioRd \left( \ioRd |z_1-z_2|^{\gamma}p_{t(\lambda_1-\lambda_2)}^{(\alpha)}(z_2,z_1)dz_1\right) \abs{V(z_2)}dz_2 \\ \nonumber
&=M \,\,||V||_1\,\,\mathbb{E}^{0}[ \,\abs{X_{t(\lambda_1-\lambda_2)}}^{\gamma}\,]\\ \nonumber
&=M\,\, ||V||_1\,\, \left(t(\lambda_1-\lambda_2)\right)^{\frac{\gamma}{\alpha}}\,\,
\mathbb{E}^{0}[\,\abs{X_1}^{\gamma}\,].
\end{align}
Thus, by using the fact that 
$$\int_{0}^{1}\int_{0}^{\lambda_1}\left(\lambda_1-\lambda_2\right)
^{\gamma/\alpha} d\lambda_2d\lambda_1=
 \left( \frac{\gamma}{\alpha}+2\right)^{-1}\left(\frac{\gamma}{\alpha}+1 \right)^{-1},$$
we obtain that  the conclusion of the lemma follows from \eqref{1remestimate}, \eqref{bridge2} by setting 
$$R(t)=2t^2 \int_{0}^{1}\int_{0}^{\lambda_1}\iogRd{2}\left(V(z_1)-V(z_2)\right) V(z_2)p_{t(\lambda_1-\lambda_2)}^{(\alpha)}(z_2,z_1)dz_1dz_2d\lambda_1d\lambda_2.$$
\end{proof}
Therefore,  using that $t^3\leq t^{2+\frac{\gamma}{\alpha}}$
 for $t\in (0,1)$, we have that  Theorem \ref{Holdercont} is a consequence  of applying   Lemma
 \ref{2ndtermestimate}  to the following expression obtained in Lemma \ref {Qbridge} when $J=2$,
\begin{align*}
Q_V^{(\alpha)}(t)= -t\ioRd V(x)dx+  
\frac{1}{2}\iogRd{2}\palp(x,y)E^{t}_{x,y}\left[\left(\int_{0}^{t}V(X_s)ds\right)^2
 \right]dxdy + R_3(t),
 \end{align*}
 where we already know that
 $$ \abs{R_3(t)}\leq t^3||V||_1||V||_{\infty}^2
 e^{t||V||_{\infty}}.$$


\section{general expansion for  rapidly decreasing smooth potential}\label{sec:smoothpot}
\label{sec:GeneralCoef}
We have already seen in the previous section that by adding an extra regularity condition on the potential $V$, namely, H\"{o}lder continuity and using $X_t=t^{1/\alpha}X_1$ in law, we have been able to extract a second term in the expansion of $Q_V^{(\alpha)}(t)$. In this section, we
 will obtain more terms and find explicit expressions for these which as before will depend on the potential $V$. 
 
 Let $V\in \mathcal{S}(\Rd)$.
 We denote by  $\wh{V}$ the Fourier transform of $V$ with the normalization 
 \begin{equation}\label{Fourierdef}
\wh{V}(\xi)= \ioRd e^{-\dot{\iota}x\cdot \xi}V(x)dx.
 \end{equation}

We note that due to our definition of  $\wh{V}$, we have, by setting 
$\bar{d}\xi=(2\pi)^{-d}d\xi$, 
\begin{enumerate}
\item[(i)] (Fourier inversion formula) 
\begin{equation*} \label{InversionForm}
V(x)=\ioRd e^{\dot{\iota}x\cdot\xi}\wh{V}(\xi)\bar{d}\xi, 
\end{equation*} 
and 
\item[(ii)] (Plancherel identity) For $f,g\in \mathcal{S}(\Rd)$,
\begin{equation*}\label{Plancherel}
\ioRd e^{-\dot{\iota}x\cdot\xi }f(x)g(x)dx= \ioRd \wh{f}(\theta)\wh{g}(\xi-\theta)\bar{d}\theta.
\end{equation*}
\end{enumerate}
The fact that $V\in \mathcal{S}(\Rd)$ will allows us to apply the inversion formula to each summand in Lemma \ref{Qbridge} which in turn will provide  the terms obtained in Theorem \ref{5termexp}.  To do this, we need the following proposition.

\begin{prop}
 For any $k\geq2$,
\begin{align}\label{inversionterm}
&\iogRd{k}\myprod{i}{1}{k}V(z_i)\myprod{j}{1}{k-1}p_{t(\lambda_{j}-\lambda_{j+1})}^{(\alpha)}(z_{j},z_{j+1})dz_k...dz_1 =\\ \nonumber
&\iogRd{(k-1)}\wh{V}(-\mysum{i}{1}{k-1}\theta_i) \myprod{i}{1}{k-1}\wh{V}(\theta_i)
\exp\left(-t\mysum{r}{1}{k-1}(\lambda_{r}-\lambda_{r+1})
\abs{\mysum{m}{1}{r}\theta_m}^\alpha \right)\bar{d}\theta_{k-1}...\bar{d}\theta_{1}.
\end{align}
\end{prop}

\begin{proof}
Under the notation given in \eqref{heavynotation} we have 
\begin{align*}
V_{k}\left(z^{(k)}\right)&=V_k(z_1,...,z_k)=\myprod{i}{1}{k}V(z_i);\\
p\left(t,z^{(k)}\right)&=\myprod{r}{1}{k-1}p_{t(\lambda_{r}-\lambda_{r+1})}^{(\alpha)}(z_{r},z_{r+1}).
\end{align*}

By applying Fourier transform in
$\R^{kd}$ and Plancherel identity, we obtain
\begin{equation}\label{Planc.Gen.Term}
\iogRd{k}V_k\left(z^{(k)}\right)p\left(t,z^{(k)}\right)dz^{(k)}=
\iogRd{k}\wh{V_k}\left(\theta^{(k)}\right)\wh{p}
\left(t,-\theta^{(k)}\right)\bar{d}\theta^{(k)}.
\end{equation}

Next, it  follows  easily that if  $\theta^{(k)}=(\theta_1,...\theta_k)$, $\theta_i \in \Rd$,  then
$$\wh{V_k}(\theta^{(k)})=\myprod{i}{1}{k}\wh{V}(\theta_i).$$

On the other hand, we claim that
\begin{equation}\label{Fourierp}
\wh{p}(t,-\theta^{(k)})= (2\pi)^d\delta\left({\tiny\mysum{i}{1}{k}}\theta_i\right)\exp\left(-t\mysum{r}{1}{k-1}(\lambda_{r}-\lambda_{r+1})
{\tiny\abs{\mysum{m}{1}{r}\theta_m}^\alpha}\right).
\end{equation}
To see this, we observe by \eqref{Fourierdef} that
\begin{align*}\label{densityRkd}
\wh{p}(t,-\theta^{(k)})=\int_{\R^{kd}}\exp\left( \dot{\iota}\sum\limits_{j=1}^{k}\xi_j\cdot\theta_j\right) \prod\limits_{r=1}^{k-1}
p_{t(\lambda_{r}-\lambda_{r+1})}^{(\alpha)}(\xi_r-\xi_{r+1})d\xi^{(k)}.
\end{align*}
By considering the substitutions $z_r=\xi_r-\xi_{r+1}$, $r\in\set{1,...,k-1}$, we have
for any $j\in \set{1,...,k-1}$ that
$$\xi_j= \xi_k+\mysum{r}{j}{k-1}z_r.$$
Therefore, we obtain after interchanging the order of summation that
\begin{equation*}
\sum\limits_{j=1}^{k-1} \xi_j\cdot \theta_j= \sum\limits_{r=1}^{k-1}z_r \cdot\,\left(\sum\limits_{m=1}^{r}\theta_{m}\right) +
 \xi_{k}\cdot\,\, \sum\limits_{i=1}^{k-1}\theta_{i}.
\end{equation*}
Thus, \eqref{Fourierp} follows by using that
\begin{align*}
\wh{p_t^{(\alpha)}}(\xi)&=e^{-t|\xi|^{\alpha}}, \\
\int_{\Rd}\exp\left(\dot{\iota}\xi_{k}\cdot\,\sum\limits_{i=1}^{k}\theta_{i}\right)d\xi_k&= (2\pi)^d\delta\left(\sum\limits_{i=1}^{k}\theta_{i}\right),
\end{align*} 
and 
\begin{eqnarray*}
\wh{p}(t,-\theta^{(k)})&=&\int_{\R^{kd}}\exp\left( \dot{\iota}\sum\limits_{r=1}^{k-1}z_r\cdot \left(\sum\limits_{m=1}^{r}\theta_{m}\right)+\dot{\iota}\xi_{k}\cdot\,\, \sum\limits_{i=1}^{k}\theta_{i}\right)\\
&&\times \prod\limits_{r=1}^{k-1}
p_{t(\lambda_{r}-\lambda_{r+1})}^{(\alpha)}(z_r)dz^{(k-1)}d\xi_k.
\end{eqnarray*}
Consequently, the conclusion of the proposition follows from   \eqref{Planc.Gen.Term} and \eqref{Fourierp}.
\end{proof}
We next recall the  Taylor expansion for the exponential function 
\begin{equation}\label{Tay.exp}
e^{-x}= \mysum{n}{0}{M}\frac{(-1)^n}{n!}x^{n} + \frac{(-1)^{M+1}}{(M+1)!}x^{M+1}e^{-x \beta_{M+1}(x)},
\end{equation}
valid for every $x\geq 0$ and integer $M\geq 0$, where we call $\beta_{M+1}(x)\in (0,1)$ the remainder of order $M+1$.

 We also recall  that for $k\geq2$ integer, the Binomial theorem asserts  that
$$(x_1 + x_2 + \cdots + x_{k-1})^n = \sum_{
\substack{(\ell_1,...,\ell_{k-1})\,\,\,\in \,\,\,\mathbb{N}^{k-1},\\ \ell_1+\ell_2+\ldots+\ell_{k-1}=n}} {n \choose \ell_1, \ell_2, \ldots, \ell_{k-1}} x_1^{\ell_1} x_2^{\ell_2} \cdots x_{k-1}^{\ell_{k-1}}.$$
Next, bearing in mind the notation given in \eqref{heavynotation}, we set $\gamma_r =\mysum{m}{1}{r}\theta_m$, 

\begin{equation*}
 \ell^{(k-1)}=(\ell_1,\,...\,,\ell_{k-1})\in \mathbb{N}^{k-1},
 \end{equation*}
 
 \begin{equation*}
A(n, \ell^{(k-1)})={n \choose \ell_1, \ell_2, \ldots, \ell_{k-1}}\int_{I_k}\myprod{i}{1}{k-1}(\lambda_i-\lambda_{i+1})^{\ell_i}d\lambda^{(k)},
\end{equation*}
and 
\begin{align}\label{Texp}
T_k(t)=\int_{I_k}\iogRd{(k-1)}&\myprod{i}{1}{k-1}\wh{V}(\theta_i)\wh{V}(-\mysum{i}{1}{k-1}\theta_i)\\ \nonumber
&\times\exp\left(-t\mysum{r}{1}{k-1}(\lambda_{r}-\lambda_{r+1})
\abs{\mysum{m}{1}{r}\theta_m}^{\alpha}\right)
\bar{d}\theta^{(k-1)}d\lambda^{(k)}.
\end{align}

Therefore, under this notation, we obtain the following expansion for the term $T_k(t)$. 
\begin{cor}\label{Term}
Let $M\geq0$ and $k\geq2$ be integers. Then
\begin{align*}
T_k(t)=\mysum{n}{0}{M}\frac{(-t)^n}{n!}C_{n,k}(V) + R_{M+1}^{(k)}(t),
\end{align*}
where
\begin{eqnarray*}
R_{M+1}^{(k)}(t)&=&\frac{(-t)^{M+1}}{(M+1)!}\int_{I_k}\iogRd{(k-1)}\myprod{i}{1}{k-1}\wh{V}(\theta_i)\wh{V}(-\mysum{i}{1}{k-1}\theta_i)\\
&&\times \left(\mysum{r}{1}{k-1}(\lambda_r-\lambda_{r+1})
|\gamma_r|^{\alpha}\right)^{M+1}e^{-\Upsilon}\bar{d}\theta^{(k-1)}
d\lambda^{(k)},
\end{eqnarray*} 
for some nonnegative function 
$\Upsilon=\Upsilon(t,\lambda^{(k)},\theta^{(k-1)},M+1)$.  The remainder  satisfies
\begin{equation}\label{Rem1}
 R_{M+1}^{(k)}(t)=\mathcal{O}(t^{M+1}),
\end{equation}
as $\tgo$. Moreover, the coefficients are given by 
\begin{eqnarray*}\label{genterm}
C_{n,k}(V)&=&\sum_{
\substack{(\ell_1,...,\ell_{k-1})\,\,\in \,\, \mathbb{N}^{k-1},\\ \ell_1+\ell_2+\ldots+\ell_{k-1}=n}} A(n,\ell^{(k-1)})
\iogRd{(k-1)}\wh{V}(-\mysum{i}{1}{k-1}\theta_i)\myprod{i}{1}{k-1}\wh{V}(\theta_i)\\
&&\times
\abs{\mysum{m}{1}{i}\theta_m}^{\alpha\ell_i}\bar{d}\theta^{(k-1)}.
\end{eqnarray*}
\end{cor}
\begin{proof}
The formula for the coefficients is obtained by applying the Taylor expansion for the  exponential function and the Binomial theorem to our expression of $T_k(t)$ in
\eqref{Texp}.

Next, we proceed to show our claim about the remainder. In order to do so, we point out that  $V\in \mathcal{S}(\Rd)$  implies that all quantities to appear below are finite. Also, the constant $C$  will depend on $k,M$ and $\alpha$ and its value may change from line to line. It is easy  to observe that for some $C>0$, we have
$$\left(\mysum{r}{1}{k-1}(\lambda_r-\lambda_{r+1})
|\gamma_r|^{\alpha}\right)^{M+1}\leq C\mymax{m}{k-1}|\theta_m|^{\alpha(M+1)}.$$
In particular, if we let $\Lambda_{r}=\set{\theta^{(k-1)}\in \R^{d(k-1)}:\mymax{m}{k-1}|\theta_m|=
|\theta_r|}$, we arrive at
\begin{align*}
\abs{ R_{M+1}^{(k)}(t)}&\leq Ct^{M+1}
\mysum{r}{1}{k-1}\int_{\Lambda_r}|\wh{V}|(-\gamma_{k-1})
|\wh{V}(\theta_r)|\theta_r|^{\alpha(M+1)}|
\prod_{i=1,i\neq r}^{k-1}|\wh{V}|(\theta_i)\bar{d}\theta^{(k-1)} \\
&\leq Ct^{M+1}||\wh{V}||_{\infty}||\wh{\F_{M+1}(V)}||_{1}
||\wh{V}||_{1}^{k-2}.
\end{align*}
Here, $\F_{M+1}$ stands for the composite of $\F$ with itself  $M+1$--times and this completes the proof.
\end{proof}
With  Corollary \ref{Term} at hand, we carry on  showing the existence of a general expansion for $Q_V^{(\alpha)}(t)$  for small time. 

\begin{thm}\label{Generalexpansion}
 For any integer $N\geq2$,
\begin{align}\label{genexpansion}
Q_V^{(\alpha)}(t)=-t\ioRd V(\theta)d\theta+\mysum{\ell}{2}{N}(-t)^{\ell}C_{\ell}(V) + \mathcal{O}(t^{N+1}),
\end{align}
as $\tgo$. Here, $$C_{\ell}(V)=
\sum\limits_{\substack{n+k=\ell \\2\leq k}}\frac{1}{n!}C_{n,k}(V),$$ 
where
$C_{n,k}(V)$ as defined  in Corollary \ref{Term}.
\end{thm}
\begin{proof}
As a result of Lemma \ref{Qbridge}, Corollary \ref{Term} and  \eqref{inversionterm}, we have for any integers 
$J\geq2$ and $M\geq 0$ that
\begin{align}\label{Quasiexpansion}
Q_V^{(\alpha)}(t)=-t\ioRd V(\theta)d\theta +\mysum{k}{2}{J}\mysum{n}{0}{M}\frac{(-t)^{k+n}}{n!}C_{n,k}(V) + R_{M+1,J+1}(t), 
\end{align}
where  
$$R_{M+1,J+1}(t)=R_{J+1}(t) +\mysum{k}{2}{J}(-t)^k R_{M+1}^{(k)}.$$ 
In other words, $R_{J+1,M+1}(t)$ is the sum of all those remainders provided by Lemma \ref{Qbridge} and Corollary \ref{Term}. We also point out that due to \eqref{Rem2} and \eqref{Rem1}, we conclude
 $$R_{M+1,J+1}(t)=\mathcal{O}\left(t^{\min\set{J+1,M+3}}\right),$$ as $\tgo$.

Since $M$ and $J$ are arbitrary, given $N\geq2$, we may choose $M$ and $J$ as large as we desire so that
$$\min\set{J+1,M+3}\geq N+1$$ and such that formula \eqref{Quasiexpansion} can be decomposed as follows

\begin{equation}\label{genexpwithrem}
Q_V^{(\alpha)}(t)=-t\ioRd V(\theta)d\theta+\sum\limits_{\substack{2\leq n+k\leq N \\2\leq k}}\frac{(-t)^{k+n}}{n!}C_{n,k}(V) + \tilde{R}_{N+1}(t),
\end{equation}
where $\tilde{R}_{N+1}(t)$ is defined to be
\begin{equation*}
 \sum\limits_{\substack{ n+k\geq N+1 \\2\leq k}}\frac{(-t)^{k+n}}{n!}C_{n,k}(V) + R_{M+1,J+1}(t) .
\end{equation*}
Thus, it is easy to observe that  $\tilde{R}_{N+1}(t)=\mathcal{O}(t^{N+1})$
as $\tgo$.

The conclusion of the theorem follows by noticing that the second terms on the right hand side of both  \eqref{genexpwithrem} and  \eqref{genexpansion} are the same under our definition of $C_{\ell}(V)$.
\end{proof}

Before proceeding, we give an application concerning the  coefficients $C_{\ell}(V)$. The corollary roughly says that we can characterize the potential $V$ from the coefficients under some extra assumptions.  This corollary should be compared to the result for the trace ($\alpha=2$ case) given in \cite[Corollary 2.1]{BanSab}. 

\begin{cor}
Let $V \in \mathcal{S}(\Rd)$ be such that $\wh{V}\geq 0$. If $C_{\ell}(V)=0$ for some $\ell \geq 2$, then we must have
$V(x)=0$ for all $x \in \Rd$.
\end{cor}

\begin{proof}
By Theorem \ref{Generalexpansion} and corollary \ref{Term}, we see that  $C_{\ell}(V)\geq 0$ for all $\ell \geq 2$ when  $\wh{V}\geq 0$. In particular, the condition
 $C_{\ell}(V)= 0$ for some  $\ell \geq 2$ implies that 
\begin{equation*} 
C_{\ell-2,2}(V)={\ell-2 \choose\ell- 2} \int_{I_2}(\lambda_1-\lambda_{2})^{\ell-2}d\lambda^{(2)}
\int_{\Rd}\wh{V}(-\theta_1)\wh{V}(\theta_1)
\abs{\theta_1}^{\alpha(\ell-2)}\bar{d}\theta_1=0.
\end{equation*}
Therefore, we must have $\wh{V}(-\theta_1)\wh{V}(\theta_1)=0$ for all $\theta_1\in \Rd$. Now by applying Plancherel
identity we have
\begin{equation*}
\int_{\Rd}|V(x)|^2dx=\int_{\Rd}\wh{V}(-\theta_1)\wh{V}(\theta_1)\bar{d}\theta_1=0
\end{equation*}
and this gives the claimed result. 
\end{proof}
\section{computation of coefficients}\label{sec:coeff}
In this section we write down explicitly  the first five coefficients of the asymptotic expansion given in  \eqref{genexpansion}.  This also proves Theorem
\ref{5termexp}. All the results in the previous section also hold  for $\alpha=2$. Therefore we will consider $0<\alpha\leq 2.$ 

In order to find the coefficients $C_3(V)$, $C_4(V)$ and $C_5(V)$,  we will resort to Lemma \ref{simpleformint} below.   We start by observing that by means of the inversion formula, it follows easily that 
\begin{equation}\label{easyterm}
C_{0,k}(V)=\frac{1}{k!}\ioRd V^k(\theta)d\theta,
\end{equation}
for any integer $k\geq2$. 

\begin{lem} \label{simpleformint}
Let $k\geq 2$ be an integer. Assume that $\set{\ell_i, i\in \set{1,..., k-1}}$ is a  sequence of nonnegative real numbers satisfying
\begin{equation}\label{lcond}
\mysum{i}{1}{k-1}\ell_i=n,
\end{equation} 
for some positive real number $n$. Then
\begin{enumerate}
\item[(a)] If $k=2$, we have
$$\int_{I_2}(\lambda_1-\lambda_{2})^{n}d\lambda^{(2)}=\frac{1}{(1+n)(2+n)}.$$
\item[(b)] If $k\geq3$, we obtain
\begin{align*}
\int_{I_k}\myprod{i}{1}{k-1}(\lambda_i-\lambda_{i+1})^{\ell_i}d\lambda^{(k)}=
\frac{1}{(k+n)(\ell_{k-1}+1)}\myprod{i}{1}{k-2}
\int_{0}^{1}(1-s)^{\ell_i}s^{ k+n-(i+1 +\mysum{j}{1}{i}\ell_j)}ds.
\end{align*}
\end{enumerate}
\end{lem}  
\begin{proof}
We only need to prove {(b)}. Let $\lambda_1\in (0,1)$ be fixed. Consider the following change of variables 
\begin{equation*}\label{cv}
\lambda_{i+1}=\lambda_i s_i,
\end{equation*}
for $i\in \set{1,...,k-1}$. Using the fact that $0<\lambda_{i+1}<\lambda_i$  we must have that $s_i\in(0,1)$.
Notice that this change of variables yields
\begin{equation}\label{elam}
\lambda_{i+1}=\lambda_1\myprod{j}{1}{i}s_j. 
\end{equation}
Thus, the Jacobian associated to this change of variables is the determinant of an upper triangular matrix and it is given explicitly by the following formula. 
$$\frac{\partial(\lambda_2,...,\lambda_{k})}{\partial(s_1,...,s_{k-1})}=\lambda_1^{k-1}\myprod{i}{1}{k-2}s_i^{k-(i+1)}.$$
Observe that by \eqref{elam} and \eqref{lcond} we have 
\begin{align*}
\myprod{i}{2}{k-1}\lambda_i^{\ell_i}&= 
\myprod{i}{2}{k-1}\left(\lambda_1\myprod{j}{1}{i-1}s_j\right)^{\ell_i}=\lambda_1^{\mysum{j}{2}{k-2}\ell_j}
\left(\myprod{i}{1}{k-2}s_i^{\mysum{j}{i+1}{k-1}\ell_j}
\right)
=\lambda_1^{n-\ell_1}\myprod{i}{1}{k-2}s_i^{n-\mysum{j}{1}{i}\ell_j}.
\end{align*}
From this we conclude that 
\begin{align*}
\myprod{i}{1}{k-1}(\lambda_i-\lambda_{i+1})^{\ell_i}&=\myprod{i}{1}{k-1}\lambda_i^{\ell_i}(1-s_i)^{\ell_i}
=\lambda_1^{\ell_1}(1-s_{k-1})^{\ell_{k-1}}
\myprod{i}{1}{k-2}(1-s_{i})^{\ell_i}\myprod{i}{2}{k-1}\lambda_i^{\ell_i} \\ 
&=\lambda_1^{n}(1-s_{k-1})^{\ell_{k-1}}
\myprod{i}{1}{k-2}(1-s_{i})^{\ell_i}\myprod{i}{1}{k-2}s_i^{n-\mysum{j}{1}{i}\ell_j}.
\end{align*}
As a result, integrating both sides of the above identity, we see that $(b)$ is a consequence of the following equality.
\begin{eqnarray*}
\int_{I_k}\myprod{i}{1}{k-1}(\lambda_i-\lambda_{i+1})^{\ell_i}d\lambda^{(k)}&=& 
\int_{0}^{1}\lambda_1^{n+k-1}d\lambda_1\int_{0}^{1}(1-s_{k-1})^{\ell_{k-1}}ds_{k-1}\\
&&\times \int_{[0,1]^{k-2}}
\myprod{i}{1}{k-2}(1-s_{i})^{\ell_i}s_i^{n+k-(i+1+\mysum{j}{1}{i}\ell_j)}ds^{(k-2)}.
\end{eqnarray*}
\end{proof}

For the computations to be performed below is worth recalling that 
\begin{eqnarray*}
\mathcal{E}_{\alpha}(V)=\int_{\Rd}\F V(\theta)V(\theta)d\theta
=\ioRd \wh{V}(-\theta)\wh{V}(\theta)|\theta_1|^{\alpha}\bar{d}\theta
=\int_{\Rd} |\wh{V}(\theta)|^2|\theta|^{\alpha}\bar{d}\theta.
\end{eqnarray*}

\begin{lem}    
\begin{equation*}
C_{3}(V)= \frac{1}{3!}\left(\ioRd V^3(\theta)d\theta +\mathcal{E}_{\alpha}(V)\right).
\end{equation*}
\end{lem}
\begin{proof}

By Theorem \ref{Generalexpansion}, we have
$$C_{3}(V)=C_{0,3}(V)+C_{1,2}(V).$$ 
From \eqref{easyterm}, it suffices to compute $C_{1,2}(V).$  
Following Corollary \ref{Term}, we have, by Plancherel Theorem and  \eqref{Fourierstable},  that
\begin{align*}
C_{1,2}(V)&=A(1,1)\ioRd \wh{V}(-\theta_1)\wh{V}(\theta_1)|\theta_1|^{\alpha}\bar{d}\theta_1 
=\frac{1}{6}\ioRd V(\theta) \F V(\theta)d\theta, 
\end{align*}
which gives the formula above. 
\end{proof}

\begin{lem}\label{4thterm}
\begin{equation*}
C_{4}(V)=\frac{1}{4!}\left(\ioRd V^4(\theta)d\theta+2\ioRd V^2(\theta)\F V(\theta)d\theta +\int_{\Rd}\abs{\F V(\theta)}^2d\theta\right).
\end{equation*}
\end{lem}
\begin{proof}
By Theorem \ref{Generalexpansion},
$$C_{4}(V)=C_{0,4}(V)+C_{1,3}(V)+\frac{1}{2!}C_{2,2}(V).$$

By Corollary \ref{Term} with $n=1$ and $k=3$, we have
\begin{align*}
C_{1,3}(V)=&A(1,(1,0))\iogRd{2}\wh{V}(\theta_2)\wh{V}(\theta_1)
\wh{V}(-\theta_1-\theta_2)|\theta_1|^{\alpha}\bar{d}\theta_2\bar{d}\theta_1 \\+
&A(1,(0,1))\iogRd{2}\wh{V}(\theta_2)\wh{V}(\theta_1)
\wh{V}(-\theta_1-\theta_2)|\theta_1+\theta_2|^{\alpha}\bar{d}\theta_2\bar{d}\theta_1.
\end{align*}

From Lemma \ref{simpleformint}, we obtain
\begin{align*}
A\left(1,(1,0)\right)&={1 \choose 1,0}\frac{1}{4}\int_{0}^1(1-s)sds=\frac{1}{4!}, \\ \nonumber
A\left(1,(0,1)\right)&={1 \choose 0,1}\frac{1}{4\cdot 2}\int_{0}^1s^2ds
=\frac{1}{4!}.
\end{align*}

On the other hand, due to the basic properties of the Fourier transform, 
\begin{align*}
&\iogRd{2}\wh{V}(\theta_2)\wh{V}(\theta_1)
\wh{V}(-\theta_1-\theta_2)|\theta_1+\theta_2|^{\alpha}\bar{d}\theta_2\bar{d}\theta_1\\
&=\int_{\Rd}\left(\int_{\Rd}\wh{V}(\theta_2)\wh{\F V}(-\theta_1-\theta_2)\bar{d}\theta_2\right)\wh{V}(\theta_1)\bar{d}\theta_1\\
&=\int_{\Rd}\left(\int_{\Rd}e^{\dot{\iota}\theta\cdot
\theta_1}V(\theta)\F V(\theta)d\theta\right)\wh{V}(\theta_1)\bar{d}\theta_1 \\
&=\ioRd V^2(\theta)\F V(\theta)d\theta.
\end{align*}

A similar argument yields
\begin{align*}
&\iogRd{2}\wh{V}(\theta_2)\wh{V}(\theta_1)
\wh{V}(-\theta_1-\theta_2)|\theta_1|^{\alpha}\bar{d}\theta_2\bar{d}\theta_1 \\ 
&=\int_{\Rd}|\theta_1|^{\alpha}\wh{V}(\theta_1)
\left(\int_{\Rd}\wh{V}(-\theta_1-\theta_2)\wh{V}(\theta_2)\bar{d}\theta_2\right)\bar{d}\theta_1 \\
&=\int_{\Rd}|\theta_1|^{\alpha}\wh{V}(\theta_1)
\left(\int_{\Rd}e^{\dot{\iota}  \theta_1\cdot \theta}V^2(\theta)d\theta\right)\bar{d}\theta_1 \\
&=\ioRd V^2(\theta)\F V(\theta)d\theta.
\end{align*}

Thus, we arrive at
\begin{align*}
C_{1,3}(V)&=\frac{2}{ 4!}\ioRd V^2(\theta)\F V(\theta)d\theta.
\end{align*}

Next,
\begin{align*}
C_{2,2}(V)&= A(2,2)
\ioRd \wh{V}(-\theta_1)\wh{V}(\theta_1)|\theta_1|^{2\alpha}\bar{d}\theta_1 
=\frac{2!}{ 4!}\int_{\Rd}\abs{\F V(\theta)}^2d\theta.
\end{align*}
Therefore, the announced formula for $C_4(V)$ follows from the above identities.
\end{proof}

\begin{lem}\label{5thterm}
\begin{align*}
C_5(V)=\frac{1}{5!}\Biggl(\int_{\Rd}V^5(\theta)d\theta &+
2\int_{\Rd}V^3(\theta)\F V(\theta)d\theta 
+2\int_{\Rd}V^2(\theta)\F_2 V(\theta)d\theta \\ 
&+\int_{\Rd}V(\theta)\abs{\F V(\theta)}^{2}d\theta + \mathcal{E}_{\alpha}\left(\F V\right)+\mathcal{E}_{\alpha}\left(V^2\right)\Biggr),
\end{align*}
where  $\F_2$ denotes the composition  of $\F$  with itself.
\end{lem}
\begin{proof}

Once again, Theorem \ref{Generalexpansion} gives 
$$C_5(V)=C_{0,5}(V)+C_{1,4}(V)+\frac{1}{2!}C_{2,3}(V)+\frac{1}{3!}C_{3,2}(V).$$ 
The first term $C_{0,5}(V)$ follows from \eqref{easyterm}. From Corollary \ref{Term} with $n=1$ and $k=4$, we have
\begin{align}\label{Diff}
C_{1,4}(V)&=A(1,\, (1,0,0))\int_{\R^{3d}}\wh{V}(-\mysum{i}{1}{3}\theta_i)\myprod{i}{1}{3}\wh{V}(\theta_i)
|\theta_1|^{\alpha}\bar{d}\theta^{(3)} \\ \label{second}
&+A(1,\, (0,1,0))\int_{\R^{3d}}\wh{V}(-\mysum{i}{1}{3}\theta_i)\myprod{i}{1}{3}\wh{V}(\theta_i)
\abs{\mysum{m}{1}{2}\theta_m}^{\alpha}\bar{d}\theta^{(3)} \\ \label{third}
&+A(1,\, (0,0,1))\int_{\R^{3d}}\wh{V}(-\mysum{i}{1}{3}\theta_i)\myprod{i}{1}{3}\wh{V}(\theta_i)
\abs{\mysum{m}{1}{3}\theta_m}^{\alpha}\bar{d}\theta^{(3)}.
\end{align}

The most difficult term to compute in the above equality is the one appearing in \eqref{second}  and we proceed to deal with this one first.  By integrating first with respect to $\theta_3$ and 
applying Plancherel formula gives 
\begin{align}\label{second1} 
&\int_{\R^{3d}}\wh{V}(-\mysum{i}{1}{3}\theta_i)\myprod{i}{1}{3}\wh{V}(\theta_i)
\abs{\mysum{m}{1}{2}\theta_m}^{\alpha}\bar{d}\theta^{(3)} \\ \nonumber 
&=\int_{\Rd}V^2(\theta)\left(\int_{\R^{2d}}
\wh{V}(\theta_1)\wh{V}(\theta_2)\abs{\theta_1+\theta_2}^{\alpha}
e^{\dot{\iota}\theta\cdot(\theta_1+\theta_2)}\bar{d}
\theta_1\bar{d}\theta_2
\right)d\theta.
\end{align}
Consider the change of variable $\theta_2=z-\theta_1$,
where the independent variable is $\theta_2$. Then, the   
integral in \eqref{second1} between parenthesis  equals
\begin{align*}
\int_{\R^{2d}}
\wh{V}(\theta_1)\wh{V}(z-\theta_1)\abs{z}^{\alpha}
e^{\dot{\iota}\theta\cdot z}\bar{d}\theta_1\bar{d}z. 
\end{align*}
Thus, integrating the last expression with respect to $\theta_1$ gives 
\begin{align*}
\int_{\Rd}\abs{z}^{\alpha}e^{\dot{\iota}\theta\cdot z}
\left(\int_{\Rd}e^{-\dot{\iota}\eta\cdot z}V^2(\eta)d\eta\right)\bar{d}z=\int_{\Rd}\abs{z}^{\alpha}
e^{\dot{\iota}\theta\cdot z}\wh{V^2}(z)\bar{d}z=
\F V^2(\theta).
\end{align*}
In other words, we have shown that
\begin{equation*}
\int_{\R^{3d}}\wh{V}(-\mysum{i}{1}{3}\theta_i)\myprod{i}{1}{3}\wh{V}(\theta_i)
\abs{\mysum{m}{1}{2}\theta_m}^{\alpha}\bar{d}\theta^{(3)}=
\mathcal{E}_{\alpha}(V^2).
\end{equation*}

Next, we claim that the other two integral terms in \eqref{third} equal 
$$\int_{\Rd}V^3(\theta)\F V(\theta)d\theta.$$ 
To see this, it suffices to consider the following equalities.
\begin{align*}
&\int_{\R^{2d}}\left(\int_{\Rd}\wh{V}(-\theta_1-\theta_2-\theta_3)\abs{\theta_1
+\theta_2+\theta_3}^{\alpha}\wh{V}(\theta_3)\bar{d}\theta_3\right)\wh{V}(\theta_1)\wh{V}(\theta_2)\bar{d}\theta_1\bar{d}\theta_2 \\
&=\int_{\R^{2d}}\left(\int_{\Rd}e^{\dot{\iota} \theta \cdot (\theta_1+\theta_2)}V(\theta)\F V(\theta)d\theta\right)
\wh{V}(\theta_1)\wh{V}(\theta_2)\bar{d}\theta_1\bar{d}\theta_2\\
&=\int_{\R^{d}}\left(\int_{\R^{2d}}e^{\dot{\iota} \theta \cdot (\theta_1+\theta_2)}\wh{V}(\theta_1)\wh{V}(\theta_2)\bar{d}\theta_1\bar{d}\theta_2 \right)
V(\theta)\F V(\theta)d\theta\\
&=\int_{\Rd}V^3(\theta)\F V(\theta)d\theta,
\end{align*}
and
\begin{align*}
&\int_{\R^{2d}}\left(\int_{\Rd}\wh{V}(-\theta_1-\theta_2-\theta_3)\wh{V}(\theta_3)\bar{d}\theta_3\right)\abs{\theta_1
}^{\alpha}\wh{V}(\theta_1)\wh{V}(\theta_2)\bar{d}\theta_1\bar{d}\theta_2 \\
&=\int_{\R^{2d}}\left(\int_{\Rd}e^{\dot{\iota} \theta \cdot (\theta_1+\theta_2)}V^2(\theta)d\theta\right)
\wh{\F V}(\theta_1)\wh{V}(\theta_2)\bar{d}\theta_1\bar{d}\theta_2\\
&=\int_{\Rd}V^3(\theta)\F V(\theta)d\theta.
\end{align*}

As far for the quantities, $A(1, \, (1,0,0))$, $A(1, \, (0,1,0))$ and $A(1, \, (0,0,1))$ we have 
\begin{align*}
A(1, \, (1,0,0))&=\frac{1}{5}\int_{0}^{1}(1-s)s^2ds
\int_{0}^{1}sds = \frac{1}{5!} \\
A(1, \, (0,1,0))&=\frac{1}{5}\int_{0}^{1}s^3ds
\int_{0}^{1}(1-s)sds= \frac{1}{5!} \\
A(1, \, (0,0,1))&=\frac{1}{5}\,\cdot\,\frac{1}{2}\,\int_{0}^{1}s^3ds\,
\int_{0}^{1}s^2ds=\,\frac{1}{5!}. 
\end{align*}
Therefore, we conclude that 
\begin{equation}\label{ONE}
C_{1,4}(V)=\frac{1}{5!}\left(2\int_{\Rd}V^3(\theta)\F V(\theta)d\theta +\mathcal{E}_{\alpha}(V^2)\right).
\end{equation}

Next, we compute $C_{2,3}(V)$.  This time we have 
\begin{align}\label{C23}
C_{2,3}(V)&=A(2, \, (1,1))\int_{\R^{2d}}\wh{V}(-\theta_1-\theta_2) \wh{V}(\theta_1)\wh{V}(\theta_2)
|\theta_1|^{\alpha}|\theta_1+\theta_2|^{\alpha}\bar{d}\theta^{(2)}\\
\nonumber
&+ A(2, \, (2,0))\int_{\R^{2d}}\wh{V}(-\theta_1-\theta_2) \wh{V}(\theta_1)\wh{V}(\theta_2)
|\theta_1|^{2\alpha}\bar{d}\theta^{(2)} \\ \nonumber
&+A(2, \, (0,2))
\int_{\R^{2d}}\wh{V}(-\theta_1-\theta_2) \wh{V}(\theta_1)\wh{V}(\theta_2)
|\theta_1+\theta_2|^{2\alpha}\bar{d}\theta^{(2)}.
\end{align}
The first integral term 
in the right hand side of above equality  equals
\begin{align*}
&\int_{\Rd}|\theta_1|^{\alpha}\wh{V}(\theta_1)
\left(\int_{\Rd}\wh{V}(-\theta_1-\theta_2)|\theta_1+\theta_2|^{\alpha}\wh{ V}(\theta_2)\bar{d}\theta_2\right)\bar{d}\theta_1 \\ \nonumber
&=\int_{\Rd}|\theta_1|^{\alpha}\wh{V}(\theta_1)
\left(\int_{\Rd}e^{\dot{\iota}\theta\cdot \theta_1}V(\theta)\F V(\theta)d\theta
\right)\bar{d}\theta_1\\ \nonumber
&=\int_{\Rd}V(\theta)\abs{\F V(\theta)}^2d\theta.
\end{align*}

As for  the other two integral terms  in
\eqref{C23}, we claim they both equal 
$$\int_{\Rd}V^2(\theta)\F_2 V(\theta)d\theta,$$
since the third integral term equals
\begin{eqnarray*}
&&\int_{\Rd}\wh{V}(\theta_1)\left(\int_{\Rd}
\wh{V}(\theta_2)\wh{ V}(-\theta_1-\theta_2)|\theta_1+\theta_2|^{2\alpha}\bar{d}\theta_2\right)\bar{d}\theta_1\\
&=&\int_{\Rd}\wh{V}(\theta_1)\left(\int_{\Rd}
e^{\dot{\iota}\theta \cdot \theta_1}V(\theta)\F_2 V(\theta)d\theta\right)\bar{d}\theta_1,
\end{eqnarray*}
whereas the second one equals
\begin{eqnarray*}
&&\int_{\Rd}|\theta_1|^{2\alpha}\wh{V}(\theta_1)\left(\int_{\Rd}
\wh{V}(\theta_2)\wh{ V}(-\theta_1-\theta_2)\bar{d}\theta_2\right)\bar{d}\theta_1\\
&=&\int_{\Rd}|\theta_1|^{2\alpha}\wh{V}(\theta_1)\left(\int_{\Rd}
e^{\dot{\iota}\theta \cdot \theta_1}V^2(\theta)d\theta\right)\bar{d}\theta_1.
\end{eqnarray*}

As for the coefficients in front of the integral terms, we have 
\begin{align*}
A(2,\, (1,1))&= {2 \choose 1,1}\frac{1}{5\cdot 2}\int_{0}^{1}(1-s)s^2ds=\frac{2}{5!} \\
A(2,\, (2,0))&= {2 \choose 2,0}\frac{1}{5}\int_{0}^{1}(1-s)^2sds= \frac{2}{5!}\\
A(2,\, (0,2))&={2 \choose 0,2}
\frac{1}{5\cdot3}\int_{0}^{1}s^3ds=\frac{2}{5!}.
\end{align*}
Therefore,
\begin{equation}\label{TWO}
C_{2,3}(V)=\frac{2!}{5!}\int_{\Rd}V(\theta)|\F V(\theta)|^2d\theta
+\frac{2\cdot 2!}{5!}\int_{\Rd}V^2(\theta)\F_2 V(\theta)d\theta.
\end{equation}

Likewise, we obtain 
\begin{align}\label{THREE}
C_{3,2}(V)&=A(3,3)
\ioRd \wh{V}(-\theta_1)\wh{V}(\theta_1)|\theta_1|^{3\alpha}\bar{d}\theta_1\\ \nonumber
&=\frac{1}{20}\ioRd \F V(\theta) \F_2 V(\theta)d\theta=\frac{3!}{5!}
\mathcal{E}_{\alpha}\left(\F V\right).
\end{align}
Combining \eqref{easyterm}, \eqref{ONE},  \eqref{TWO}, and \eqref{THREE}, we obtain our expression for $C_5(V)$. 
\end{proof}

In the case of the Laplacian  $\alpha=2$, the signs of the coefficients can be use to give information on the poles on the metomorphic extension of the resolvent of the operator $H_V$; see for example \cite[Theorem 4.1]{BanSab}. In particular, it is shown in \cite{BanSab}  that the first five coefficients in the trace expansion are non-negative provided the potential is non-negative.  Our computations above yield a similar result for the first five coefficients of the heat content.  More precisely we have

\begin{cor}\label{nonegativefive}  Suppose $V \in \mathcal{S}(\Rd)$, $V\geq 0$.   Then $C_{\ell}(V) \geq 0$, for $1\leq \ell \leq 5$. 
\end{cor} 
\begin{proof}  With 
$$
C_1{(V)}=\int_{\Rd}V(\theta)d\theta, \,\,\, C_2(V)=\frac{1}{2}\int_{\Rd}V^2(\theta)d\theta,
$$
and 
$$
C_{3}(V)= \frac{1}{3!}\left(\ioRd V^3(\theta)d\theta +\mathcal{E}_{\alpha}(V)\right), 
$$
the assertion trivially holds for these coefficients.  

We can rewrite the expression in Lemma \ref{4thterm} for $C_4(V)$ as 
\begin{eqnarray*}
C_{4}(V)&=&\frac{1}{4!}\ioRd \left(V^4(\theta)+2V^2(\theta)\F V(\theta) +\abs{\F V(\theta)}^2\right)d\theta\\
&=& \frac{1}{4!}\ioRd\abs{V^2(\theta)+\F V(\theta)}^2d\theta
\end{eqnarray*}
and this shows that $C_4(V)\geq 0$.  

For $C_5(V)$,  we re-group the expression given by Lemma \ref{5thterm} as follows.
\begin{eqnarray*}
C_5(V)
&=&\frac{1}{5!}\int_{\Rd}V(\theta)\left(V^4(\theta)+
2V^2(\theta)\F V(\theta)+\abs{\F V(\theta)}^{2}\right)d\theta\\
&+&\frac{1}{5!}\Biggl(\mathcal{E}_{\alpha}\left(\F V\right)+2\int_{\Rd}V^2(\theta)\F_2 V(\theta)d\theta+\mathcal{E}_{\alpha}\left(V^2\right)\Biggr)\\
&=&\frac{1}{5!}\int_{\Rd}V(\theta)\abs{V^2(\theta)+\F V(\theta)}^{2}d\theta\\
&+&\frac{1}{5!}\Biggl(\mathcal{E}_{\alpha}\left(\F V\right)+2\int_{\Rd}V^2(\theta)\F_2 V(\theta)d\theta+\mathcal{E}_{\alpha}\left(V^2\right)\Biggr).
\end{eqnarray*}
If $V$ is non-negative the first of the last two terms above is clearly nonnegative.  We claim the last term is also non-negative.  To show this,  we  use Plancherel's identity for the second term and write the Dirichlet form in terms of the Fourier transform. However, we need to be a little careful here since the Fourier transform of a real valued function may be complex valued. Below we write $Re(z)$  for the real part of the complex number  $z$ and use 
the fact that for real valued functions, $\wh{V}(-\theta)=\overline{\wh{V}(\theta)}$.   We write 
\begin{eqnarray*}
2\int_{\Rd}V^2(\theta)\F_2 V(\theta)d\theta&=&\int_{\Rd}\wh{V^2}(-\theta)|\theta|^{2\alpha}\wh{V}(\theta)\bar{d}\theta+\int_{\Rd}\wh{V^2}(\theta)|\theta|^{2\alpha}\wh{V}(-\theta)\bar{d}\theta\\
&=& \int_{\Rd}\overline{\wh{V^2}}(\theta)|\theta|^{2\alpha}\wh{V}(\theta)\bar{d}\theta+\int_{\Rd}\wh{V^2}(\theta)|\theta|^{2\alpha}\overline{\wh{V}}(\theta)\bar{d}\theta\\
&=& 2\int_{\Rd}|\theta|^{2\alpha}Re\left(\wh{V^2}(\theta){\wh{V}(\theta)}\right)\bar{d}\theta. 
\end{eqnarray*}
Similarly, 
\begin{align*}
\mathcal{E}_{\alpha}\left(V^2\right)=\int_{\Rd}|\theta|^{\alpha}|\wh{V^2}(\theta)|^2\bar{d}\theta \,\,\,\,\, 
\mbox{and} \,\, \,\,\,
\mathcal{E}_{\alpha}\left(\F V\right)=\int_{\Rd}|\theta|^{3\alpha}|\wh{V}(\theta)|^2\bar{d}\theta.
\end{align*}
Putting these identities together gives 
\begin{eqnarray*}
&&\mathcal{E}_{\alpha}\left(\F V\right)+2\int_{\Rd}V^2(\theta)\F_2 V(\theta)d\theta+\mathcal{E}_{\alpha}\left(V^2\right)
\\
&=&\int_{\Rd}\left( |\theta|^{2\alpha}|\wh{V}(\theta)|^2+2|\theta|^{\alpha}
Re\left(\wh{V^2}(\theta){\wh{V}(\theta)}\right) +|\wh{V^2}(\theta)|^2
\right)|\theta|^{\alpha}\bar{d}\theta\\
&=&\int_{\Rd}\abs{|\theta|^{\alpha}\wh{V}(\theta)+\wh{V^2}(\theta) }^2 |\theta|^{\alpha} \bar{d}\theta . 
\end{eqnarray*}
This together with our previous estimate shows that $C_5(V)\geq 0$, for  $V\geq 0$. 
\end{proof}

\begin{rmk}  It is interesting to observe  that for all $V \in \mathcal{S}(\Rd)$ (regardless of the sign), $C_2(V)$ and $C_4(V)$ are nonnegative.  Whether or not this pattern remains as we move up along the even integers is an interesting question.  With some patience   one may be able to test this for $C_6(V)$ and perhaps even $C_8(V)$ but the general term is not clear at all.
\end{rmk}

The probabilistic and Fourier transform techniques of this paper have been used recently in \cite{Acu2}  to prove the existence of decompositions for additive functionals for one dimensional Cauchy  and relativistic Cauchy stable processes.  For more on this line of work, we refer the interested reader to \cite{Acu2}, \cite{Nua}
and references therein.

\end{document}